\documentclass[reqno, 11pt]{amsart}

\usepackage[letterpaper,hmargin=1in,vmargin=1in]{geometry}
\usepackage{amsmath, amssymb, amsthm, verbatim, url}
\usepackage{graphicx}
\usepackage{enumerate,pinlabel}
\usepackage{amsmath,amssymb,amsthm,mathrsfs,graphicx,url}
\usepackage[usenames,dvipsnames]{color}
\usepackage[colorlinks=true,linkcolor=Black,citecolor=Black, urlcolor=Black]{hyperref}
\numberwithin{equation}{section}

\theoremstyle{plain}
\newtheorem{theorem}{Theorem}
\newtheorem*{theorem*}{Theorem}
\newtheorem{lemma}{Lemma}[section]
\newtheorem{proposition}{Proposition}[section]

\theoremstyle{definition}

\theoremstyle{remark}

\newcommand\supp{\mathop{\rm supp}}
\newcommand\real{\mathop{\rm Re}}
\newcommand\imag{\mathop{\rm Im}}
\newcommand*{\defeq}{\mathrel{\vcenter{\baselineskip0.5ex \lineskiplimit0pt
                     \hbox{\scriptsize.}\hbox{\scriptsize.}}}
                     =}

\title{Local energy decay for Lipschitz wavespeeds}
\author{Jacob Shapiro}

\begin{document}
\begin{abstract}
We prove a logarithmic local energy decay rate for the wave equation with a wavespeed that is a compactly supported Lipschitz perturbation of unity. The key is to establish suitable resolvent estimates at high and low energy for the meromorphic continuation of the cutoff resolvent. The decay rate is the same as that proved by Burq for a smooth perturbation of the Laplacian outside an obstacle.
\end{abstract}
\maketitle 
\author

\section{Introduction} \label{introduction}

\noindent The purpose of this article is to obtain a local energy decay rate for the wave equation \begin{equation} \label{prelim wave}
 \begin{cases} \partial_t^2u -c^2(x)\Delta u = 0 , &(x,t) \in \mathbb{R}^n \times (0, \infty), \\    u(x, 0) = u_0(x), \\ \partial_tu(x,0) = u_1(x), \end{cases}
\end{equation} 
where $\Delta \le 0$ is the Laplacian on $\mathbb{R}^n$, $n \ge 2$. We assume the initial data is compactly supported with $\nabla u_0 \in (H^1(\mathbb{R}^n))^n$ and $u_1 \in H^1(\mathbb{R}^n)$. The wavespeed $c$ is a compactly supported Lipschitz perturbation of unity, see \eqref{basic c assumptions} below.

For $x \in \mathbb{R}^n$ and $R > 0$, set $B(x, R) \defeq \{y \in \mathbb{R}^n : |y - x| < R\}$. We obtain the following logarithmic local energy decay.
\begin{theorem} \label{brief decay}
Assume 
\begin{equation} \label{basic c assumptions}
 c = c(x) >0, \quad  c, c^{-1} \in L^\infty(\mathbb{R}^n), \quad \nabla c \in (L^\infty(\mathbb{R}^n))^n, \quad \supp(c-1) \text{ is compact}.
\end{equation}
Suppose the supports of $u_0$ and $u_1$ are contained in $B(0,R_1)$, and that $\nabla u_0 \in (H^1(\mathbb{R}^n))^n$ and $u_1 \in H^1(\mathbb{R}^n)$.
 Then for any $R_2 > 0$, there exists $C > 0$ such that the solution $u$ to \eqref{prelim wave} satisfies for $t \ge 0$,
\begin{equation} \label{prelim decay}
\left( \int_{B(0, R_2)} |\nabla u|^2 + c^{-2} |\partial_t u|^2 dx \right)^{\frac{1}{2}} \le \frac{C}{\log(2 + t)}\left( \| \nabla u_0 \|_{(H^1(\mathbb{R}^n))^n} + \| u_1\|_{H^1(\mathbb{R}^n)} \right).
\end{equation}
\end{theorem}
Theorem \ref{brief decay} follows from Theorem \ref{decay}, which appears in Section \ref{main statement} below. In Theorem \ref{decay}, we obtain additional powers of $\log(2+t)$ in the denominator if $u_0$ and $u_1$ possess greater regularity with respect to the differential operator $-c^2(x)\Delta$. 

In contrast with local energy decay, the global energy of the solution to \eqref{prelim wave} is conserved because the wave propagator is unitary, see \eqref{wave prop unitary} in section \ref{main statement}.  


The decay rate  \eqref{prelim decay} was first obtained by Burq \cite{bu98,bu02} for smooth perturbations of the Laplacian outside an obstacle. Bouclet \cite{bo11} established a similar decay rate on $\mathbb{R}^n$ when the Laplacian is defined by an asymptotically Euclidean metric. For logarithmic decay rates in transmission problems and general relativity, see \cite{be03, ga17, mo16}. The novel aspect of Theorem \ref{brief decay} is that \eqref{prelim decay} now holds with a weaker regularity condition on the wavespeed.

Logarithmic decay rates are well-known to be optimal when resonances are exponentially close to the real axis. This connection was observed by Ralston \cite{ra69}. He later showed that such resonances exist for a certain class of smooth wavespeeds \cite{ra71}. See \cite{hosm14} for a related construction in general relativity.



The study of local energy decay more broadly has a long history which we will not review here. Some recent papers using techniques similar to those in this article include \cite{povo99} and \cite{ch09}. See also \cite{hizw17} for more historical background and references. 



To prove Theorem \ref{decay}, the key is to establish suitable Sobolev space estimates at high and low energy on the norm of the  meromorphic continuation of the cutoff resolvent $\chi R(\lambda) \chi \defeq \chi (-c^2 \Delta - \lambda^2)^{-1} \chi$,  where $\chi \in C_0^\infty(\mathbb{R}^n)$ and $\lambda \in \mathbb{R} \setminus \{0\}$. Here, the relevant spaces are $L^2(\mathbb{R}^n)$ and $\dot{H}^1(\mathbb{R}^n)$, a homogeneous Sobolev space. They correspond to the second and first terms on the left side of \eqref{prelim decay}, respectively. 

At high energy, we use the exponential semiclassical resolvent estimate \eqref{semiclassical cutoff} for a Lipschitz potential, proved by Datchev \cite{da} and the author \cite{s}, to show (Proposition \ref{perturbed resolv est high energy})
\begin{gather}
\| \chi R(\lambda) \chi \|_{L^2(\mathbb{R}^n)\to L^2(\mathbb{R}^n)} \le e^{C_1| \lambda|}, \qquad   \lambda \in \mathbb{R} \setminus [-M,M], \text{ some $M > 1$}. \label{high energy est}
\end{gather}
At low energy, we find (Proposition \ref{desired exp prop})
\begin{gather}
\| \chi R(\lambda) \chi \|_{L^2(\mathbb{R}^n) \to L^2(\mathbb{R}^n)} \le C_1 (1 + |\lambda |^{n-2} |\log \lambda|) , \qquad  \lambda \in [-\varepsilon_0, \varepsilon_0] \setminus \{0\}, \text{ some $0 < \varepsilon_0< 1$}. \label{bound near 0}
\end{gather}

Since the first posting of this paper as a preprint, there has been further progress in the setting of low regularity semiclassical resolvent estimates. Klopp and Vogel \cite{klvo18} and the author \cite{sh18}, working independently and at the same time, showed a different exponential resolvent estimate for semiclassical Schr\"dinger operators with a compactly supported $L^\infty$ potential. Using arguments like those in Section \ref{resolvent estimate at high energy} below, this semiclassical estimate implies a bound similar to \eqref{high energy est}, while \eqref{bound near 0} is unaffected by reducing the regularity. As a result, one has a different logarithmic decay rate for wavespeeds $c$ that are only an $L^\infty$ perturbation of unity.  

\begin{theorem} \label{brief decay infty}
Assume
\begin{equation*} 
 c = c(x) >0, \quad  c, c^{-1} \in L^\infty(\mathbb{R}^n), \quad \supp(c-1) \text{ is compact}.
\end{equation*}
Suppose $\supp u_0$, $\supp u_1 \subseteq B(0,R_1)$, $\nabla u_0 \in (H^1(\mathbb{R}^n))^n$, and $u_1 \in H^1(\mathbb{R}^n)$. Then for any $ \varepsilon >0$ and $R_2 > 0$, there exists $C > 0$ such that the solution $u$ to \eqref{prelim wave} satisfies for $t \ge 0$,
\begin{equation} \label{prelim decay infty}
\left( \int_{B(0, R_2)} |\nabla u|^2 + c^{-2} |\partial_t u|^2 dx \right)^{\frac{1}{2}} \le \frac{C}{\log^{\frac{3}{4 + \varepsilon}}(2 + t)}\left( \| \nabla u_0 \|_{(H^1(\mathbb{R}^n))^n} + \| u_1\|_{H^1(\mathbb{R}^n)} \right).
\end{equation}
\end{theorem}

See  \cite[Chapter 5]{sh18b} for proof of Theorem \ref{brief decay infty} using techniques similar to the ones in this paper. 

The main technical innovation in this article is the careful distinction between $\dot{H}^1(\mathbb{R}^n)$ and $\dot{H}^1(B(0,R))$, the space of elements of $\dot{H}^1(\mathbb{R}^n)$ supported in a fixed ball, to deduce from \eqref{high energy est} and \eqref{bound near 0} analogous estimates for the homogeneous space. 

If $n\ge 3$, one can extend the continuation of $\chi R(\lambda) \chi$  to a bounded operator on all of $\dot{H}^1(\mathbb{R}^n)$ using that for any $\chi \in C_0^\infty(\mathbb{R}^n)$, there exists $C_\chi > 0$ such that  
\begin{equation} \label{Poincare dim two}
\| \chi \varphi \|_{L^2(\mathbb{R}^n)} \le C_\chi \| \nabla \varphi \|_{L^2(\mathbb{R}^n)}, \qquad \text{all $\varphi \in C_0^\infty(\mathbb{R}
^n)$.} 
\end{equation}
This estimate follows, for instance, from the Gagliardo–-Nirenberg–-Sobolev inequality \cite[Theorem 1, Section 5.6.1]{ev}. 

On the other hand, \eqref{Poincare dim two} fails when $n = 2$, creating an obstruction to extending $\chi R(\lambda) \chi$  to all of $\dot{H}^1(\mathbb{R}^n)$. However, for any $R_1 > 0$ as in Theorem \ref{brief decay}, restricting to $C_0^\infty(B(0,R_1))$ restores access to \eqref{Poincare dim two}, with $C_\chi$ now also depending  on $R_1$. Then, for any dimension $n \ge 2$, the continuation of $\chi R(\lambda) \chi$ extends as a bounded operator $\dot{H}^1(B(0,R_1)) \to \dot{H}^1(\mathbb{R}^n)$ with norm estimates similar to \eqref{high energy est} and \eqref{bound near 0}, which are sufficient to prove Theorem \ref{decay}. 

The logarithmic singularity appearing in \eqref{bound near 0} when $n =2$  differs from the case of an obstacle, where the resolvent is bounded near zero in all dimensions. Although, this singularity is still weak enough to allow integral estimates via Stone's Formula, similar to the those appearing in \cite{povo99}. From these estimates  we conclude Theorem \ref{decay}.


The outline of this article is as follows. In Section \ref{main statement}, we give the more general statement of the local energy decay. In Section \ref{background}, we state preliminary facts about scattering theory and about the Hilbert space $H$ on which we define  our wave equation. In Sections \ref{resolvent estimate at low energy} and \ref{resolvent estimate at high energy}, we prove the $L^2(\mathbb{R}^n) \to L^2(\mathbb{R}^n)$ cutoff resolvent estimates at high and low energy, and in Section \ref{statement of main resolvent estimate}, we convert them into the appropriate $\dot{H}^1(B(0,R_1)) \to \dot{H}^1(\mathbb{R}^n)$ resolvent estimate. Finally, in Section \ref{proof of local energy decay}, we combine this latter resolvent estimate with Stone's Formula to prove the local energy decay.  

For the reader's convenience and the sake of completeness, we include an appendix in which we prove the operators $L$ and $B$, defined in Section \ref{main statement}, are self-adjoint.  

The author is grateful to Kiril Datchev for  helpful discussions and suggestions during the writing of this article, and to the Purdue Research Foundation for support through a research assistantship.

\section{Statement of the local energy decay} \label{main statement}

In this section, we state Theorem \ref{decay}, the main theorem in the paper. 
We begin by setting up the functional analytic framework in which we work. 

Let $\Omega \subseteq \mathbb{R}^n$ be open. Set $L^2_c(\Omega) \defeq L^2(\Omega, c^{-2}(x)dx)$, where $c$ satisfies \eqref{basic c assumptions}. Let $\dot{H}^1(\Omega)$ denote the homogeneous Sobolev space of order one, defined as the Hilbert completion of $C^\infty_0(\Omega)$ with respect to the norm
\begin{equation*}
\|\varphi\|^2_1 \defeq \int_{\Omega} |\nabla \varphi(x)|^2 dx. 
\end{equation*}
Thus, elements of $\dot{H}^1(\Omega)$ are equivalence classes $[\varphi_m]$ of sequences $\{ \varphi_m \} \subseteq C_0^\infty(\Omega)$ which are Cauchy with respect to the $\| \cdot \|_1$-norm. For an element $u = [\varphi_m] \in \dot{H}^1(\Omega)$, we denote by $\nabla u$ the vector which is the limit in $(L^2(\Omega))^n$ of the vectors $\nabla \varphi_m$.

Because the non-homogeneous Sobolev space $H^1(\mathbb{R}^n)$ is the completion of $C_0^\infty(\mathbb{R}^n)$ with respect to a stronger norm, by inclusion we may regard $H^1(\mathbb{R}^n)$ as a closed subspace of $\dot{H}^1(\mathbb{R}^n)$. Also, for any $\Omega \subseteq \mathbb{R}^n$, the inclusion map $C_0^\infty(\Omega) \to C_0^\infty(\mathbb{R}^n)$ induces an isometry $\dot{H}^1(\Omega) \to \dot{H}^1(\mathbb{R}^n)$. So we may also regard $\dot{H}^1(\Omega)$ as a closed subspace of $\dot{H}^1(\mathbb{R}^n)$. 

We note, for the sake of completeness, that, for $n \ge 2$, $\dot{H}^1(\mathbb{R}^n)$ may be regarded as a set of translation classes of functions $u \in H^1_{\text{loc}}(\mathbb{R}^2)$ such that $\nabla u \in L^2(\mathbb{R}^2)$, equipped with the inner product $(u ,v) \mapsto \langle \nabla u, \nabla v \rangle_{(L^2)^n}$. See \cite{orsu12} for further details. 

We work within the Hilbert space $H \defeq \dot{H}^1(\mathbb{R}^n) \oplus L^2_c(\mathbb{R}^n)$. For $R > 0$, set 
\begin{equation*}
H_R \defeq \{(u_0, u_1) \in H : u_0 \in \dot{H}^1(B(0,R)), \text{ }\supp u_1 \subseteq B(0,R) \}.
\end{equation*}
This is a closed subspace of $H$ and is the space of initial conditions on which we show the local energy decay holds. 

Set $L \defeq -c^2(x) \Delta : L^2_c(\mathbb{R}^n) \to L^2_c(\mathbb{R}^n)$, which is nonnegative and self-adjoint with respect to the domain $D(L) = H^2(\mathbb{R}^n)$. Define the operator $B$ by the matrix
\begin{equation*}
B \defeq \begin{bmatrix} 0 & iI \\ -iL  & 0 \end{bmatrix}: H \to H,
\end{equation*}
which is self-adjoint with respect to the domain 
\begin{equation*}
D(B) \defeq \{(u_0,u_1) \in H : \Delta u_0 \in L^2(\mathbb{R}^n),  u_1 \in H^1(\mathbb{R}^n)\}.
\end{equation*}
In the Appendix, we prove that $L$ and $B$ are self-adjoint on their respective domains. 

For $k \in \mathbb{N}$, let $\| \cdot \|_{D(B^k)}$ be the graph norm associated to $B^k$:
\begin{equation*}
\|(u_0,u_1) \|_{D(B^k)} \defeq \| (u_0,u_1) \|_H + \|B^k(u_0,u_1)\|_H, \qquad (u_0,u_1) \in D(B^k).
\end{equation*}


The operator $B$ allows us to write the wave equation as a first order system. That is, given $(u_0, u_1) \in H$, 
\begin{equation} \label{wave prop unitary}
U(t) \defeq (U_0(t), U_1(t)) = e^{-itB}(u_0, u_1),
\end{equation}
is the unique solution in $H$ to the wave equation
\begin{equation} \label{wave equation c}
 \begin{cases} \partial_t U + iBU  = 0, \qquad \text{in } \mathbb{R}^n \times (0,\infty), \\ U(0) = (u_0, u_1).\\
 \end{cases}  
\end{equation}
We now state the local energy decay rate for solution of \eqref{wave equation c}.
\begin{theorem} \label{decay}
Suppose that $(u_1, u_0) \in D(B^k) \cap H_{R_1}$ for some $k \in \mathbb{N}$ and $R_1 > 0$. Then for any $R_2 >0$, there exists $C > 0$ such that for $t \ge 0$,
\begin{equation} \label{decay rate}
\left(\int_{B(0, R_2)} |\nabla U_0|^2(t) + |U_1|^2(t) dx \right)^{\frac{1}{2}} \le \frac{C}{(\log(2 +t))^k} \|(u_0, u_1)\|_{D(B^k)}.
\end{equation}
\end{theorem}

\section{Background and preliminaries} \label{background}

\noindent In this section, we recall several facts about the analytic continuation of the cutoff resolvent for the free Laplacian. We also explain why, when the resolvent is perturbed to include the wavespeed, the continuation still has no real poles away from zero. Finally, we describe how the homogeneous Sobolev space of a ball behaves with respect to the perturbed resolvent as well as the wave operator. The proofs in subsequent sections rely on these facts.

\subsection{Continuation of the free resolvent} 

Set $\mathbb{R}_{+} \defeq \{ r \in \mathbb{R} : r \ge 0\}$, $\mathbb{R}_{-} \defeq \{ r \in \mathbb{R} : r \le 0\}$. If $\imag \lambda >0$, we have $\lambda^2 \notin \mathbb{R}_{+}$, hence $R_0(\lambda) \defeq (-\Delta - \lambda^2)^{-1}$ is well-defined as a bounded operator $L^2(\mathbb{R}^n) \to H^2(\mathbb{R}^n)$. For $\chi \in C_0^\infty(\mathbb{R}^n)$, it is well-known that the free cutoff resolvent 
\begin{equation*}
\chi R_0(\lambda) \chi = \chi(-\Delta - \lambda^2)^{-1} \chi : L^2(\mathbb{R}^n) \to H^2(\mathbb{R}^n),
\end{equation*}
continues analytically from $ \imag \lambda > 0$ to $\mathbb{C}$ when $n \ge 2$ is odd and to $\mathbb{C} \setminus i\mathbb{R}_{-} $ when $n$ is even. In fact, in even dimensions, the continuation can be made to the logarithmic cover of $\mathbb{C}\setminus\{0\}$, although we will not need this stronger fact.

Furthermore, the continuation of $\chi R_0(\lambda) \chi$ has the expansion
\begin{equation} \label{free resolv E expansion}
\chi R_0(\lambda) \chi =  E_1(\lambda) + \lambda^{n-2}(\log \lambda) E_2(\lambda),
\end{equation}
for $\lambda \in \mathbb{C} \setminus {i \mathbb{R}_-}$. Here, $E_1(\lambda)$ and $E_2(\lambda)$ are entire operator-valued functions, and $E_2 \equiv 0$ when $n$ is odd.  For further details on the continuation of the cutoff resolvent, see chapters 2 and 3 of \cite{dyzw} and section 1.1 in  \cite{vod01}.

\subsection{Estimates for the continued free resolvent}

Next, we recall well-known $L^2(\mathbb{R}^n) \to L^2(\mathbb{R}^n)$ estimates for the cutoff resolvent away from the origin. In Section \ref{resolvent estimate at high energy}, we use these estimates to establish a bound on the perturbed resolvent at high energy. 
\begin{gather}  
\forall M > 0, \quad \exists C_M > 0: \quad \text{if $|\real \lambda|\ge M,$ $\imag \lambda \ge -M,$ and $|\alpha_1| + |\alpha_2| \le 2$, then } \nonumber \\ 
\| \partial_x^{\alpha_1} \chi R_0(\lambda) \chi \partial_x^{\alpha_2} \|_{L^2 \to L^2} \le C_M|\lambda|^{|\alpha_1|+|\alpha_2|-1}.  \label{nonsemiclassical resolv est away zero}
\end{gather} 
Using the Cauchy formula with \eqref{nonsemiclassical resolv est away zero} implies, for a different constant $\tilde{C}_M > 0$, 
\begin{equation}
 \label{nonsemiclassical resolv est away zero Cauchy} 
\left \|\frac{d}{d\lambda}  \partial_x^{\alpha_1} \chi R_0(\lambda) \chi \partial_x^{\alpha_2} \right \|_{L^2 \to L^2} \le \tilde{C}_M|\lambda|^{|\alpha_1|+|\alpha_2|-1}, \qquad |\real \lambda|\ge M, \text{ } \imag \lambda > -M, \text{ } |\alpha_1| + |\alpha_2| \le 2.
\end{equation}
See also \cite[Section 5]{vod14}.
\subsection{Continuation of the perturbed resolvent} \label{Continuation of the perturbed resolvent}

Set $R(\lambda) \defeq (L - \lambda^2)^{-1}$, where initially we take $\imag \lambda >0 $. For $\chi \in C_0^\infty(\mathbb{R}^n)$, the cutoff resolvent $\chi R(\lambda) \chi$ satisfies the assumptions of the black box scattering framework introduced in \cite{sjzw91} and also presented in \cite{dyzw, sj}. This implies that $\chi R(\lambda) \chi$ continues meromorphically $L^2(\mathbb{R}^n) \to H^2(\mathbb{R}^n)$ from $\imag \lambda >0 $ to $\mathbb{C} \setminus\{0\}$ when $n \ge 2$ is odd, and to $\mathbb{C}\setminus i \mathbb{R}_{-}$ when $n$ is even. As in the case of the free resolvent, the continuation in even dimensions can be made to the logarithmic cover of $\mathbb{C} \setminus \{0\}$, although this stronger result is not needed for our purposes.

It is also follows that if $\lambda \in \mathbb{R} \setminus \{0\}$ is  a pole of the continuation, then there must exist an embedded eigenvalue corresponding to $\lambda$. That is, there exists a nonzero function $u \in H^2_{\text{comp}}(\mathbb{R}^n)$ such that $(L- \lambda^2) u = 0$. For more details, see Theorems 4.17 and 4.18 in \cite{dyzw}. However, a Carleman estimate \cite[Lemma 3.31]{dyzw} rules out the possibility of embedded eigenvalues on $\mathbb{R}\setminus \{0\}$.  Therefore, the continuation of $\chi R(\lambda) \chi$ has no poles there. 

\subsection{Operators on the homogeneous Sobolev space of a ball} \label{extend resolv to homg}

If $\lambda^2 \notin \mathbb{R}_+$, there is a constant $C_\lambda $ depending on $\lambda$ such that 
\begin{equation} \label{elliptic L2 to H2}
\|R(\lambda) \varphi \|_{H^2} \le C_\lambda \| \varphi \|_{L^2}, \qquad \varphi \in C_0^\infty(\mathbb{R}^n).
\end{equation} This follows by first noting that $(1 -\Delta)R(\lambda)$ is well defined as an operator $L^2(\mathbb{R}^n) \to L^2(\mathbb{R}^n)$ because $D(L) = H^2(\mathbb{R}^n)$, and then by rewriting
\begin{equation*}
\begin{split}
(1 - \Delta) R(\lambda) &= R(\lambda) + c^{-2}LR(\lambda) \\
&= R(\lambda) + c^{-2}\left(I + \lambda^2 R(\lambda)\right),
\end{split}
\end{equation*}
which shows that $(1 - \Delta) R(\lambda)$ is in fact bounded $L^2(\mathbb{R}^n) \to L^2(\mathbb{R}^n)$. 


Furthermore, if the support of $\varphi$ is required to lie in a fixed ball $B(0,R)$, there is a Poincar\'e-type inequality for all $n \ge 2$,
\begin{equation} \label{Poincare}
\| \varphi \|_{L^2} \le C_ R \| \nabla \varphi \|_{L^2}, \qquad \varphi \in C^\infty_0(B(0,R)),
\end{equation}
where $C_R \to \infty$ as $R \to \infty$. 

Having \eqref{elliptic L2 to H2} and \eqref{Poincare} allow us to extend $R(\lambda) :L^2(\mathbb{R}^n) \to H^2(\mathbb{R}^n)$ to a bounded operator $\dot{H}^1(B(0,R)) \to H^2(\mathbb{R}^n)$ by setting
\begin{equation} \label{define resolv on homg}
R(\lambda) [\varphi_m] \defeq R(\lambda) \left( L^2\text{-}\lim \varphi_m \right), \qquad [\varphi_m] \in \dot{H}^1(B(0, R)), \qquad \lambda^2 \notin \mathbb{R}_+, 
\end{equation} 
where $L^2$-$\lim \varphi_m$ denotes the $L^2(\mathbb{R}^n)$-limit of $\{\varphi_m\}$, which exists on account of \eqref{Poincare}.


Another fact we deploy in Section \ref{proof of local energy decay} is that if $(u_0, u_1) = ([\varphi_m], u_1) \in D(B) \cap H_R$, then $B(u_0, u_1) = (iu_1, -iLu_0) \in H_{R'}$ for any $R' > R$. To show this, first observe that since $u_1 \in H^1(\mathbb{R}^n)$ and $\supp u_1 \subseteq B(0,R)$, $u_1$ may be approximated in $H^1(\mathbb{R}^n)$ by $C_0^\infty(\mathbb{R})$-functions with supports contained in $B(0, R') \supset B(0, R)$. Therefore $u_1 \in \dot{H}^1(B(0,R'))$. To see that $\supp \Delta u_0 \subseteq B(0,R')$, we integrate against $\varphi \in C_0^\infty(\mathbb{R}^n \setminus \overline{B(0,R)})$ and apply integration by parts twice. We may then take advantage of the fact that each $\supp \varphi_m \subseteq B(0,R)$, 
\[\begin{split}
\int_{\mathbb{R}^n \setminus \overline{B(0,R)}} \Delta u_0 \varphi &= -\int_{\mathbb{R}^n \setminus \overline{B(0,R)}} \nabla u_0  \cdot \nabla \varphi \\
&= -\lim_{m \to \infty} \int_{\mathbb{R}^n \setminus \overline{B(0,R)}} \nabla \varphi_m \cdot \nabla \varphi \\
&= \lim_{m \to \infty} \int_{\mathbb{R}^n \setminus \overline{B(0,R)}} \varphi_m  \Delta \varphi \\
& = 0.
\end{split}\]


\section{Resolvent estimate at low energy}\label{resolvent estimate at low energy}

The purpose of this section is to combine the expansion \eqref{free resolv E expansion} for the free cutoff resolvent with a remainder argument to establish the following low energy bound for the perturbed cutoff resolvent.

\begin{proposition} \label{desired exp prop}
Suppose that $\chi \in C_0^\infty(\mathbb{R}^n)$. Then there exists an $0 < \varepsilon_0  < 1 $ so that  
 
\begin{equation*}
\chi R(\lambda) \chi : L^2(\mathbb{R}^n) \to H^2(\mathbb{R}^n),
 \end{equation*}
is analytic in $Q_{\varepsilon_0} \defeq \{ \lambda \in \mathbb{C}: |\real \lambda |, \text{ } | \imag \lambda | \le \varepsilon_0 \} \setminus i\mathbb{R}_{-}$. Furthermore, there exists $C >0$ such that  
\begin{equation} \label{peturbed est near zero}
\| \chi R(\lambda) \chi \|_{L^2 \to L^2} \le C(1 + |\lambda|^{n-2}|\log \lambda  |), \qquad \lambda \in Q_{\varepsilon_0}.
\end{equation}
\end{proposition}


\begin{proof}
It suffices to take $\chi \equiv 1$ on the support of $c - 1$. Initially, for $\imag \lambda >0$, define $K(\lambda) : L^2(\mathbb{R}^n) \to L^2(\mathbb{R}^n)$ by
\[\begin{split}
K(\lambda) &\defeq (1 - c^{-2})\lambda^2 (-\Delta - \lambda^2)^{-1} \\
& = (1 - c^{-2}) \lambda^2 \chi (-\Delta - \lambda^2)^{-1}.
\end{split}\]
The continuation of $\chi R_0(\lambda) \chi$ then provides a continuation for $K(\lambda)\chi$ to $\mathbb{C} \setminus i\mathbb{R}_-$. From \eqref{free resolv E expansion}, we see that
\begin{equation*}
K(\lambda) \chi = (1 - c^{-2}) \lambda ^2 (E_1(\lambda) + \lambda^{n-2}\log \lambda E_2(\lambda)).
\end{equation*}
This implies that there exists $0 < \varepsilon_0 < 1$ sufficiently small so that
\begin{equation} \label{K bound}
\lambda \in Q_{\varepsilon_0} \implies \| K(\lambda) \chi \|_{L^2 \to L^2} < \frac{1}{2}.
\end{equation}
Therefore, $I + K(\lambda)\chi$ can be inverted by a Neumann series for $\lambda \in Q_{\varepsilon_0}$,
\begin{equation*}
(I + K(\lambda)\chi)^{-1}  = \sum_{n = 0}^\infty (-1)^n (K(\lambda) \chi)^n : L^2(\mathbb{R}^n) \to L^2(\mathbb{R}^n).
\end{equation*}
Furthermore, $(I + K(\lambda)\chi)^{-1}$ is analytic in $Q_{\varepsilon_0}$ because the series converges locally uniformly there. 

To proceed, notice that $(1 - \chi) K(\lambda) \equiv 0$ for $\imag \lambda >0$ because $(1 - \chi)(1 - c^{-2}) \equiv 0$. From this, it follows that, when $\imag \lambda >0$, $(I - K(\lambda)(1 - \chi))$ is both a left and right inverse for $(I + K(\lambda)(1 - \chi))$. Additionally, observe that
\begin{equation*}
I + K(\lambda) = (I + K(\lambda)(1 - \chi))(I + K(\lambda)\chi), \qquad \imag \lambda > 0. 
\end{equation*} 
Putting the two facts together, we get a left and right inverse for $I + K(\lambda)$
\begin{equation*} 
(I + K(\lambda))^{-1} = (I + K(\lambda) \chi)^{-1}(I - K(\lambda) (1 - \chi)), \qquad \imag \lambda > 0, \text{ } \lambda \in Q_{\varepsilon_0}.
\end{equation*}
We can now write for, $\imag \lambda >0$, $\lambda \in Q_{\varepsilon_0}$,
\[ \begin{split}
\chi(-\Delta - c^{-2}\lambda^2)^{-1} \chi &= \chi(-\Delta - \lambda^2)^{-1}(I + K(\lambda))^{-1}\chi \\
&= \chi(-\Delta - \lambda^2)^{-1}(I + K(\lambda)\chi)^{-1} (I - K(\lambda)(1 - \chi))\chi \\
&=  \chi(-\Delta - \lambda^2)^{-1}(I + K(\lambda)\chi)^{-1} ((I + K(\lambda)\chi) - K(\lambda))\chi \\
&=  \chi R_0(\lambda) \chi -\chi(-\Delta - \lambda^2)^{-1}(I + K(\lambda)\chi)^{-1}K(\lambda)\chi \\
&= \chi R_0(\lambda) \chi - \chi (-\Delta - \lambda^2)^{-1} \sum_{n=0}^\infty (-1)^n \left(K(\lambda) \chi \right)^{n+1} \\
&= \chi R_0(\lambda) \chi - \chi (-\Delta - \lambda^2)^{-1} K(\lambda)\chi \sum_{n=0}^\infty (-1)^n \left(K(\lambda) \chi \right)^{n} \\
&= \chi R_0(\lambda) \chi - \chi R_0(\lambda) \chi K(\lambda)\chi \sum_{n=0}^\infty (-1)^n \left(K(\lambda) \chi \right)^{n} \\
&= \chi R_0(\lambda) \chi \left( I - \sum_{n=0}^\infty (-1)^n \left(K(\lambda) \chi \right)^{n+1} \right). \\
\end{split} \]
For the second-to-last equality, we use $K(\lambda) = \chi K(\lambda)$. We see that the left side continues analytically to $Q_{\varepsilon_0}$ because the right side does. 

To finish the proof, observe that 
\begin{equation*}
\| \chi R_0(\lambda) \chi \|_{L^2 \to L^2} \le C(1 +  |\lambda|^{n-2}|\log  \lambda| ), \qquad \lambda \in Q_{\varepsilon_0},
\end{equation*}
according to \eqref{free resolv E expansion}. Moreover, it follows from \eqref{K bound} that
\begin{equation*}
\left\| I - \sum_{n=0}^\infty (-1)^n \left(K(\lambda) \chi \right)^{n+1} \right\|_{L^2 \to L^2} \le 3, \qquad \lambda \in Q_{\varepsilon_0}.
\end{equation*}
We now conclude \eqref{peturbed est near zero} because
\begin{equation*}
\chi(-\Delta - c^{-2}\lambda^2)^{-1} \chi = \chi R(\lambda) \chi c^2.
\end{equation*}
\end{proof}

\section{Resolvent estimate at high energy} \label{resolvent estimate at high energy}
\noindent The goal of this section is to establish an exponential bound on the perturbed cutoff resolvent when $| \real \lambda|$ is large. Specifically, we prove the following.

\begin{proposition} \label{perturbed resolv est high energy}
For each $\chi \in C_0^\infty(\mathbb{R}^n)$, there exist constants $C_1, C_2 >0$, $M > 1$ such that the cutoff resolvent $\chi R(\lambda) \chi$ continues analytically from $\imag \lambda >0$ into the set   $\{\lambda \in \mathbb{C}: |\real \lambda | > M, | \imag \lambda| < e^{-C_2 |\real \lambda |} \}$, where it satisfies the bound
\begin{equation} \label{perturbed resolv est high energy eq}
\| \chi R(\lambda) \chi \|_{L^2 \to L^2} \le e^{C_1 |\real \lambda|}.
\end{equation}
\end{proposition}
To prove Proposition \ref{perturbed resolv est high energy}, we use Lemmas \ref{continue to ball} and \ref{semiclassical resolv est} below. Recall from Section \ref{Continuation of the perturbed resolvent} that the continued resolvent $\chi R(\lambda) \chi$ has no poles on $\mathbb{R} \setminus \{0\}$. Lemma \ref{continue to ball} asserts that, if there  is an exponential resolvent bound on the real axis at high energy, then the continued resolvent is in fact analytic in exponentially small strips below the real axis. This is a is a non-semiclassical version of a continuation argument of Vodev \cite[Theorem 1.5]{vod14}.

\begin{lemma} \label{continue to ball}
Let $\chi \in C_0^\infty(\mathbb{R}^n)$. Suppose that there exist $C > 0$ and $M > 1$ such that whenever  $\lambda_0 \in \mathbb{R} \setminus [-M, M]$, the continuation of $\chi R(\lambda) \chi$ from $\imag \lambda > 0$ to $\mathbb{C}\setminus i \mathbb{R}_-$ satisfies
\begin{equation} \label{est on axis}
\| \chi R(\lambda_0)\chi \|_{L^2 \to L^2} \le e^{C|\lambda_0 |}.
\end{equation}
Then there exist $C_1, C_2 >0$ such that for each $\lambda_0 \in \mathbb{R} \setminus [-M,M]$, the continued cutoff resolvent is analytic in the disk $D_{\lambda_0}(e^{-C_2 | \lambda_0 |})$, where it has the estimate
\begin{equation} \label{desired est}
\|\chi R(\lambda) \chi \|_{L^2 \to L^2} \le e^{C_1|\lambda_0|}.
\end{equation}
\end{lemma} 
\begin{proof}
Let $\chi_1 \in C_0^\infty(\mathbb{R}^n)$ have the property that $\chi_1 \equiv 1$ on the support of $c-1$. Without loss of generality, we may assume that $\chi \equiv 1$ on the support of $\chi_1$. For $\imag \lambda$, $\imag \mu >0$, we have the resolvent identity
\begin{gather}
R(\lambda) - R(\mu) = (\lambda^2 - \mu^2)R(\lambda) R(\mu) \implies \nonumber \\
R(\lambda) - R(\mu) = (\lambda^2 - \mu^2)R(\lambda) \chi_1(2 - \chi_1) R(\mu) +(\lambda^2 - \mu^2)R(\lambda) (1 - \chi_1)^2 R(\mu), \label{helper 1}
\end{gather}
The first equality implies the second because $(1 - \chi_1)^2 + \chi_1(2 - \chi_1) = 1$. 

We also compute
\begin{equation} \label{helper 2}
R(\lambda)(1 - \chi_1) - (1 - \chi_1) R_0(\lambda) = R(\lambda) [\chi_1, \Delta] R_0(\lambda), \qquad \imag \lambda >0,
\end{equation}
\begin{equation} \label{helper 3}
(1 - \chi_1)R(\mu) - R_0(\mu)(1 - \chi_1) = R_0(\mu) [\Delta, \chi_1] R(\lambda), \qquad \imag \mu >0.
\end{equation}
Using \eqref{helper 1}, \eqref{helper 2}, and \eqref{helper 3}, we express $\chi R(\lambda) \chi - \chi R(\mu)\chi$ as a sum of five operators which we denote by $T_k(\lambda, \mu)$, $k = 1, \dots ,5$. 
\begin{equation}
\begin{aligned} \label{sum of five terms}
 \chi R(\lambda) \chi - \chi R(\mu)\chi & = (\lambda^2 - \mu^2) (\chi R(\lambda) \chi) (\chi_1(2 - \chi_1)) (\chi R(\mu) \chi)
\\& + (1 -\chi_1) \left[ \chi R_0(\lambda) \chi - \chi R_0(\mu) \chi \right](1 - \chi_1) 
\\& +(1 - \chi_1) \left[ \chi R_0(\lambda) \chi - \chi R_0(\mu) \chi \right]\left[\Delta, \chi_1\right](\chi R(\mu) \chi) 
\\& -(\chi R(\lambda) \chi)\left[\Delta, \chi_1\right] \left[ \chi R_0(\lambda) \chi - \chi R_0(\mu) \chi \right](1 - \chi_1) 
\\&  -(\chi R(\lambda) \chi)\left[\Delta, \chi_1\right] \left[ \chi R_0(\lambda) \chi - \chi R_0(\mu) \chi \right]\left[\Delta, \chi_1\right](\chi R(\mu) \chi).\\
& = \sum_{k=1}^5T_k(\lambda, \mu).
\end{aligned}
\end{equation}
This formula continues to hold after continuing both $\lambda$ and $\mu$ to $\mathbb{C} \setminus i \mathbb{R}$.

For $z \in \mathbb{C}$ and $r > 0$, let $D_r(z)$ denote the disk $\{ w \in \mathbb{C}: |w - z| < r\}$. To proceed, take  $\mu = \lambda_0$. We  bound the $L^2(\mathbb{R}^n) \to L^2(\mathbb{R}^n)$ norm of each $T_k(\lambda, \lambda_0)$ for $ \lambda \in D_{\lambda_0}(e^{-C_2|\lambda_0|})$, where the precise value of $C_2 > 0$ will be determined at the end of the proof. Suppose that $ \lambda \in D_{\lambda_0}(e^{-C_2|\lambda_0|})$ is not a pole of $\chi R(\lambda) \chi$. Using \eqref{nonsemiclassical resolv est away zero Cauchy}  along with the fundamental theorem of calculus for line integrals, we have, for $|\alpha_1| + |\alpha_2| \le 2$, 
\begin{equation*}
\begin{aligned} 
\|   \partial_x^{\alpha_1} \chi R_0(\lambda) \chi \partial_x^{\alpha_2}  -  \partial_x^{\alpha_1} \chi  R_0(\lambda_0) \chi \partial_x^{\alpha_2} &\|_{L^2 \to L^2 } \le \\ 
&C_M |\lambda - \lambda_0| \sup_{|\lambda - \lambda_0| < e^{-C_2 |\lambda_0|}}{|\lambda|^{ |\alpha_1| + |\alpha_2|-1}}, \quad \lambda \in D_{\lambda_0}(e^{-C_2|\lambda_0|}). 
\end{aligned}
\end{equation*}
Therefore, for some $K >0$ large enough,
\begin{equation} 
\label{difference of cutoff resolvents lambda}
\|   \partial_x^{\alpha_1} \chi R_0(\lambda) \chi \partial_x^{\alpha_2}  - \partial_x^{\alpha_1} \chi R_0(\lambda_0)\chi \partial_x^{\alpha_2} \|_{L^2 \to L^2 } \le |\lambda - \lambda_0 | e^{K|\lambda_0|}, \qquad \lambda \in D_{\lambda_0}(e^{-C_2|\lambda_0|}).
\end{equation}

Using \eqref{nonsemiclassical resolv est away zero}, \eqref{difference of cutoff resolvents lambda}, and further increasing $K >0$ if necessary, we conclude that for $\lambda \in D_{\lambda_0}(e^{-C_2|\lambda_0|})$
\begin{gather}  
 \|T_k(\lambda, \lambda_0)\|_{L^2 \to L^2} \le  |\lambda - \lambda_0 |e^{K|\lambda_0|}, \quad k = 2,3, \nonumber \\ 
 \|T_k(\lambda, \lambda_0)\|_{L^2 \to L^2} \le  |\lambda - \lambda_0 |e^{K|\lambda_0|} \| \chi R(\lambda) \chi \|_{L^2 \to L^2}, \quad k = 1,4,5. \nonumber  
\end{gather}
Hence, by \eqref{sum of five terms} we arrive at  
\begin{equation*}
\|\chi R(\lambda) \chi \|_{L^2 \to L^2} \le 3| \lambda - \lambda_0| e^{K|\lambda_0|} \|\chi R(\lambda) \chi \|_{L^2 \to L^2} + 2e^{K |\lambda_0|}.
\end{equation*}

Now, require $C_2$ to be large enough so that 
\begin{equation*}
3|\lambda - \lambda_0|e^{K|\lambda_0|} < \frac{1}{2},  
\end{equation*}
in which case there is a $C_1 > 0$ so that 
\begin{equation*}
\|\chi R(\lambda) \chi \|_{L^2 \to L^2} < e^{C_1 |\lambda_0|}, \qquad \lambda \in D_{\lambda_0}(e^{-C_2|\lambda_0|}).
\end{equation*}
We have shown then, that $\chi R(\lambda) \chi$ is uniformly bounded in $D_{\lambda_0}(e^{-C_2|\lambda_0|})$ when $\lambda$ is not a pole. Therefore, we conclude that $\chi R(\lambda) \chi$ has no poles in $D_{\lambda_0}(e^{-C_2|\lambda_0|})$.\\
\end{proof}

With Lemma \ref{continue to ball} now in hand, we just need to show \eqref{est on axis}, which will complete the proof of Proposition \ref{perturbed resolv est high energy}. To establish \eqref{est on axis}, we convert an estimate for the semiclassical cutoff resolvent
\begin{equation*}
\chi(-h^2 \Delta + V -  z)^{-1}\chi, \quad V \in L^\infty_{\text{comp}}(\mathbb{R}^n), \quad h > 0,
\end{equation*}
appearing in \cite{da,s}, into a suitable statement about $\chi R(\lambda) \chi$.

Essentially, these lemmas convert results about the semiclassical cutoff resolvent 
 
 \begin{lemma} \label{semiclassical resolv est}
Let $n \ge 2$. Suppose that $V \in L^\infty_{\text{\emph{comp}}}(\mathbb{R}^n)$ is a real-valued function such that $\nabla V$, defined in the sense of distributions, belongs to  $(L^\infty(\mathbb{R}^n))^n$. Let $E' > 0$ and $\chi \in C_0^\infty(\mathbb{R}^n)$ be fixed. Let $\delta > 0$ so that $[E' - \delta , E' + \delta] \subseteq (0, \infty)$. Then there exist constants $C, h_0 > 0$ so that
\begin{equation} \label{semiclassical cutoff} \|\chi (-h^2\Delta + V - E - i\varepsilon)^{-1} \chi\|_{L^2 \to L^2} \le e^{\frac{C}{h}}
\end{equation}
for all $E \in  [E' - \delta , E' + \delta]$, $h \in (0, h_0]$, and $\varepsilon >0$.
 
 \end{lemma}

By setting $V_c \defeq 1 - c^{-2} \in L_{\text{comp}}^\infty(\mathbb{R}^n)$ and identifying $h = |\real \lambda|^{-1}$, we translate \eqref{semiclassical cutoff}
into estimates for $\chi R(\lambda) \chi$ when $|\real \lambda|$ is large.  
\begin{proof}[Proof of \eqref{est on axis}]

 Set $V_c \defeq 1- c^{-2}$ and $ \mathcal{O} \defeq \{ \lambda \in \mathbb{C} : \real \lambda \neq 0, \text{ } \imag \lambda > 0 \}$. Without loss of generality, take $\chi \equiv 1$ on $\supp V_c$.  Define on $\mathcal{O}$ the following families of operators $L^2(\mathbb{R}^n) \to L^2(\mathbb{R}^n)$ with domain $H^2(\mathbb{R}^n)$,
\begin{gather}
A(\lambda) \defeq -(\real \lambda)^{-2} \Delta + V_c + (\imag \lambda)^2(\real \lambda)^{-2} c^{-2} - i2 \imag \lambda(\real \lambda)^{-1}c^{-2} - 1, \nonumber \\
B(\lambda) \defeq -(\real \lambda)^{-2} \Delta + V_c - i2 \imag \lambda(\real \lambda)^{-1} + (\imag \lambda)^2(\real \lambda)^{-2} - 1, \nonumber 
\end{gather}
Furthermore, define on $\mathcal{O}$ the family $L^2(\mathbb{R}^n) \to L^2(\mathbb{R}^n)$,
\begin{equation*}
D(\lambda) \defeq (\imag \lambda)^2(\real \lambda)^{-2}V_c - i2 \imag \lambda(\real \lambda)^{-1}V_c.
\end{equation*}

We first subtract, 
\begin{equation*}
B(\lambda) - A(\lambda) = D(\lambda).
\end{equation*}
Composing with inverses, we get
\begin{equation*} 
A(\lambda)^{-1} - B(\lambda)^{-1} = B(\lambda)^{-1} D(\lambda)   A(\lambda)^{-1} \implies( I -  B(\lambda)^{-1}D(\lambda))A(\lambda)^{-1} = B(\lambda)^{-1},
\end{equation*}
Multiplying on the left and right by $\chi$ and noticing that $D(\lambda) = \chi D(\lambda) \chi$, we arrive at
\begin{equation} \label{prelim resolv id A and B}
(I - \chi B(\lambda)^{-1} \chi D(\lambda)) \chi A(\lambda)^{-1} \chi = \chi B(\lambda)^{-1} \chi, \qquad \lambda \in \mathcal{O}. 
\end{equation}

Setting $E'= 1$ and $h =|\real \lambda|^{-1}$, we apply  Lemma \ref{semiclassical resolv est} to $B(\lambda)^{-1}$. This gives $M, C, \delta > 0$ so that 
\begin{equation}\label{est for B}
\| \chi B(\lambda)^{-1} \chi \|_{L^2 \to L^2} \le e^{C |\real \lambda|}, \qquad |\real \lambda| > M, \text{ } 0 < |\imag \lambda| < \delta .
\end{equation}
Moreover, by decreasing $\delta$ if necessary, it holds that 
\begin{equation} \label{est for D}
\|D(\lambda)\|_{L^2 \to L^2} < \frac{1}{2}e^{-C|\real \lambda|}, \qquad |\real \lambda| > M, \text{ }  0 < |\imag \lambda| < \delta .
\end{equation}
Therefore, we can invert $(I - \chi B(\lambda)^{-1} \chi D(\lambda))$ by a Neumann series. From \eqref{prelim resolv id A and B}, \eqref{est for B}, and \eqref{est for D} we get 
\begin{equation} \label{Neumann A and B}
\chi A(\lambda)^{-1} \chi = \left(\sum_{k=0}^\infty (\chi B(\lambda)^{-1} \chi D(\lambda))^k \right) \chi B(\lambda)^{-1} \chi, \quad  |\real \lambda | > M, \text{ } 0< |\imag \lambda| <  \delta.  
\end{equation}
Finally, we notice that
\begin{equation*}
\chi R(\lambda) \chi = (\real \lambda)^{-2}  \chi A(\lambda)^{-1}  \chi c^{-2}, \qquad \lambda \in \mathcal{O}.
\end{equation*}
Then \eqref{est on axis} follows from the estimates \eqref{est for B} and \eqref{est for D} along with the identity \eqref{Neumann A and B}.\\
\end{proof}


\section{Statement of the main resolvent estimate} \label{statement of main resolvent estimate}

The objective in this section is to prove Proposition \ref{continuation by comp}. It states that when the resolvent $R_B(\lambda)$ acts on initial data in $H_R$, it continues analytically from $\imag \lambda >0$ to a region in the lower half plane with estimates on the norm there. These properties follow from the resolvent estimates proved for $\chi R(\lambda) \chi: L^2 \to L^2$ in the previous two sections.

To keep our notation manageable, we set
\begin{equation*}
(L_R^2)^{n+1} \defeq (L^2(B(0,R))^n \oplus L_c^2(B(0,R)). 
\end{equation*}
  Define $S_R : H \to (L_R^2)^{n+1} $ by $S_R(u_0, u_1) = (\nabla u_0 , u_1)$. Note $\| S_R \|_{H \to (L_R^2)^{n+1} } =1$. Also, throughout this section, $a \lesssim b$ means that $ a \le Cb$ for some $ C >0$ that does not depend on $\lambda$. 

\begin{proposition} \label{continuation by comp}
Let $R_1, R_2 >0$. There exist $C_1, C_2 > 0$, $M > 1$, and $0 < \varepsilon_0 < 1$ so that for all $(u_0, u_1) \in H_{R_1}$, 
$S_{R_2}R_B(\lambda)(u_0,u_1)$ continues analytically from $\imag \lambda > 0$  to the region
\begin{equation} \label{combined res free region}
\begin{aligned}
\Theta \defeq &\{ \lambda \in \mathbb{C} : |\real \lambda | > M, \text{ }\imag \lambda > - e^{-C_2|\real \lambda|} \} \cup \\ 
&\{ \lambda \in \mathbb{C} : 0 < |\real \lambda| \le M, \text{ }  \imag \lambda > - e^{C_2M}\}.
\end{aligned}
\end{equation}
One possible $\Theta$ is depicted in Figure \ref{resfree}. Furthermore, $S_{R_2}R_B(\lambda)(u_0,u_1)$ obeys the estimate 
\begin{equation} \label{RB est}
\|S_{R_2}R_B(\lambda)(u_0, u_1) \|_{(L_R^2)^{n+1}} \lesssim  \begin{cases} e^{C_1|\real \lambda|} & \lambda \in \Theta \cap \{ |\real \lambda | > \varepsilon_0\}, \\
 1 + |\lambda|^{n-2}|\log \lambda| & \lambda \in \Theta \cap \{ 0< |\real \lambda | \le \varepsilon_0\} .
\end{cases}   
\end{equation} 
\end{proposition}
\begin{figure}[h]
\labellist
\small 
\pinlabel $M$  at 88 20
\pinlabel $-M$  at 16 20
\pinlabel $\varepsilon_0$ at 50 24
\pinlabel $\imag\lambda\text{ $=$ }-e^{-C_2|\real \lambda |}$ at 105 10
\endlabellist
\includegraphics[width =12cm]{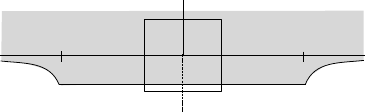}
 \caption{The region $\Theta$.}
 \label{resfree}
\end{figure}

To prove Proposition \ref{continuation by comp}, we first make a compactness argument to show that we may combine the resolvent estimates of Propositions \ref{desired exp prop} and \ref{perturbed resolv est high energy} to obtain a version of \eqref{RB est} for $\chi R(\lambda) \chi : L^2(\mathbb{R}^n) \to L^2(\mathbb{R}^n)$.

\begin{lemma} \label{piece together resolvent estimates}
For each $\chi \in C_0^\infty(\mathbb{R}^n)$, there exist $C_1, C_2 > 0$, $M > 1$, and $0 < \varepsilon_0 < 1$ such that the meromorphic continuation of the cutoff resolvent $\chi R(\lambda) \chi : L^2(\mathbb{R}^n) \to L^2(\mathbb{R}^n)$ has no poles in the region $\Theta$ of \eqref{combined res free region}, where it obeys 
\begin{equation} \label{cutoff resolv final est}
\|\chi R(\lambda) \chi \|_{L^2 \to L^2} \lesssim  \begin{cases} e^{C_1|\real \lambda|} & \lambda \in \Theta \cap \{ |\real \lambda | > \varepsilon_0\}, \\
 1 + |\lambda|^{n-2}|\log \lambda|  &\lambda \in \Theta \cap \{ 0< |\real \lambda | \le \varepsilon_0\}.
\end{cases}   
\end{equation} 
\end{lemma}

\begin{proof}
Let $\varepsilon_0$ be as in the statement of Proposition \ref{desired exp prop}. Let $C_1$, $C_2$, and $M$ be as in the statement of Proposition \ref{perturbed resolv est high energy eq}. 

Set $\tau \defeq \min\{\varepsilon_0, e^{-C_2 M}\}$. There exist only finitely many poles of $\chi R(\lambda) \chi$ in the compact set $\{ \lambda \in \mathbb{C} : \varepsilon_0 \le |\real \lambda | \le M, -\tau \le \imag \lambda \le 0  \}$. Furthermore, as discussed in Section \ref{Continuation of the perturbed resolvent}, there are no poles of $\chi R(\lambda) \chi$ on the strips $\{\lambda \in \mathbb{R} : \varepsilon_0 \le|\lambda| \le M \}$. Therefore, there exists $0 < \tau' \le \tau$ so that $\{ \lambda \in \mathbb{C} : \varepsilon_0 \le |\real \lambda | \le M, -\tau' \le \imag \lambda \le 0  \}$ contains no poles of $\chi R(\lambda) \chi$. If we redefine $M = -(\log \tau')/C_2$, then $\chi R(\lambda) \chi$ has no poles in \eqref{combined res free region} 


Using  \eqref{peturbed est near zero}, \eqref{perturbed resolv est high energy eq}, and the continuity of $\chi R(\lambda) \chi$ on the rectangles $\{\lambda \in \mathbb{C}: \varepsilon_0 \le |\real \lambda | \le M, \text{ } e^{-C_2M} \le -\imag \lambda \le 1\}$, we get

\begin{equation} \label{piece one}
\|\chi R(\lambda) \chi \|_{L^2 \to L^2} \lesssim  \begin{cases} e^{C_1|\real \lambda|} & |\real \lambda | > M, \text{ } |\imag \lambda | \le e^{-C_2 |\real \lambda|}, \\
e^{C_1|\real \lambda|} & \varepsilon_0 < | \real \lambda | \le M , \text{ } -e^{-C_2M} \le \imag \lambda  \le 1, \\
 1 + |\lambda|^{n-2}|\log \lambda|  &0< |\real \lambda | \le \varepsilon_0, \text{ } \hspace{.05cm} -e^{-C_2M} \le \imag \lambda \le 1.
\end{cases} 
\end{equation}

To finish showing  \eqref{cutoff resolv final est}, we invoke the spectral theorem, which says that for $\imag \lambda > 0$,   
\begin{equation*}
\|\chi  R(\lambda) \chi  \|_{H \to H} \lesssim \frac{1}{\text{dist}(\lambda^2 , \mathbb{R}_+)}.
\end{equation*}
The above bound implies, for instance,
\begin{equation} \label{piece two}
\|\chi  R(\lambda) \chi  \|_{H \to H} \lesssim \begin{cases} (\imag \lambda )^{-1} \le e^{C_2|\real \lambda|}, & |\real \lambda| > M, \text{ }  \imag \lambda > e^{-C_2|\real \lambda|}\\ 1 & |\real \lambda | \le M, \text{ } \imag \lambda > 1. \end{cases}
\end{equation}
Piecing together \eqref{piece one} and \eqref{piece two}, we arrive at \eqref{cutoff resolv final est}.\\
\end{proof}

Recall in section \ref{extend resolv to homg} we extended $R(\lambda)$ for $\lambda^2 \notin \mathbb{R}_+$ to a bounded operator $\dot{H}^1(B(0,R)) \to H^2(\mathbb{R}^n) $ using \eqref{define resolv on homg}. Along with this, we now define bounded $I_1 : \dot{H}^1(B(0,R)) \to H^1(\mathbb{R}^n)$ by 
\begin{equation*}
I_1[\varphi_m] \defeq L^2 \text{-} \lim \varphi_m, \qquad [\varphi_m] \in \dot{H}^1(B(0,R)).
\end{equation*}
The estimate \eqref{Poincare} shows that the above limit function exists and belongs to $H^1(\mathbb{R}^n)$.

Using these operators, we build the bounded matrix operator
  $\mathcal{M}_R(\lambda) :H_R \to H^2(\mathbb{R}^n) \oplus H^1(\mathbb{R}^n) \subseteq  H$,
\begin{equation} \label{defn of matrix op}
\mathcal{M}_R(\lambda) \begin{bmatrix} u_0\\u_1 \end{bmatrix} \defeq \begin{bmatrix} \lambda R(\lambda) & iR(\lambda) \\ -i\lambda^2 R(\lambda) - iI_1 & \lambda R(\lambda)   \end{bmatrix} \begin{bmatrix} u_0 \\ u_1 \end{bmatrix}, \qquad \imag \lambda > 0, 
\end{equation}
where $R(\lambda)$ acts on $u_1$ as the usual resolvent sending $L^2(\mathbb{R}^n) \to H^2(\mathbb{R}^n)$. A brief calculation shows that for all $(u_0,u_1) \in H_R$
\begin{equation*}
\mathcal{M}_R(\lambda) \begin{bmatrix} u_0 \\ u_1 \end{bmatrix} \in D(B) \quad \text{and} \quad (B- \lambda)\mathcal{M}_R(\lambda) \begin{bmatrix} u_0 \\ u_1 \end{bmatrix} = \begin{bmatrix} u_0 \\ u_1 \end{bmatrix}, \qquad \imag \lambda > 0.
\end{equation*}
Therefore we conclude
\begin{equation} \label{resolv matrix id wave op}
R_B(\lambda)(u_0,u_1) = \mathcal{M}_R(\lambda) \begin{bmatrix} u_0 \\ u_1 \end{bmatrix}   \qquad (u_0, u_1) \in H_R, \qquad \imag \lambda > 0.
\end{equation}
Now that we have the estimate \eqref{cutoff resolv final est} and the identity \eqref{resolv matrix id wave op}, we can prove Proposition \ref{continuation by comp}. 
\begin{proof}[Proof of Proposition \ref{continuation by comp}]

Let $\chi \in C_0^\infty(\mathbb{R}^n)$ with $\chi \equiv 1$ on $\overline{B(0,R_1)} \cup \overline{B(0,R_2)}$. For $(u_0, u_1) = ([\varphi_m], u_1) \in H_{R_1}$, set $\varphi = L^2\text{-}\lim \varphi_m$. Note that $\chi \varphi = \varphi$. Also, for any function $u \in H^1(\mathbb{R}^n)$, $\nabla u = \nabla (\chi u)$ as vectors in $(L^2(B(0,R_2)))^n$. Combining these observations with \eqref{defn of matrix op} and \eqref{resolv matrix id wave op}, we get 
\begin{equation} \label{matrix id put cutoffs}
S_{R_2} {R_B}(\lambda) (u_0,u_1) = \begin{bmatrix} \lambda \nabla  \chi R(\lambda)  \chi  \varphi + i\nabla  \chi R(\lambda)  \chi u_1 \\ -i \lambda^2  \chi R(\lambda) \chi \varphi -i \varphi + \lambda  \chi R(\lambda)  \chi u_1 \end{bmatrix}, \qquad  \imag \lambda >0.
\end{equation}
 
By Lemma \ref{piece together resolvent estimates}, the entries in the second component of the right side of  \eqref{matrix id put cutoffs} continue analytically from $\imag \lambda > 0$ to \eqref{combined res free region}. Their $L^2(B(0,R_2))$-norms have estimates of the form \eqref{cutoff resolv final est} for a possibly larger constant $C_1$, to account for the factors of $\lambda$ that appear. 

The terms in the first component continue analytically to \eqref{combined res free region} by the identity 
\begin{equation} \label{gradient identity with cutoffs}
\nabla \chi R(\lambda) \chi = \nabla \chi R_0(\lambda) \chi +  \nabla \chi \tilde{\chi} R_0(\lambda) \tilde{\chi} (1-c^{-2})(\chi + \tilde{\chi} R(\lambda) \tilde{\chi} \chi).
\end{equation}
where $\tilde{\chi} \in C_0^\infty(\mathbb{R}^n)$ is identically one on $\supp \chi \cup \supp (1-c^{-2})$.
The bounds \eqref{nonsemiclassical resolv est away zero} and \eqref{cutoff resolv final est}  imply $\| \nabla \chi R(\lambda) \chi \|_{L^2 \to (L^2)^n}$ also has a bound of the form \eqref{cutoff resolv final est}, where again we  may need to increase $C_1$. Because we have shown each component of $S_{R_2} \mathcal{M}R_B(\lambda) (u_0 ,u_1)$ obeys an estimate of the form \eqref{cutoff resolv final est}, the triangle inequality ensures that \eqref{RB est} holds. \\ 
\end{proof}

We collect one additional fact before proving the local energy decay in the next section. By the spectral theorem, $R(\lambda)^* = R(\overline{\lambda})$, $\lambda^2 \notin \mathbb{R}_+$. Therefore, when $\imag \lambda < 0$, we have the identities
\begin{equation} \label{adjoint formulas}
\begin{aligned} 
\|\chi R(\lambda) \chi \|_{L^2 \to L^2} & = \|(\chi R(\lambda) \chi )^* \|_{L^2 \to L^2}  \\
& = \| \chi R(\overline{\lambda}) \chi \|_{L^2 \to L^2},  \\
\|\partial_x^\alpha \chi R_0(\lambda) \chi \|_{L^2 \to L^2} & = \|(\partial_x^\alpha \chi R_0(\lambda) \chi )^* \|_{L^2 \to L^2}  \\
& = \| \chi R_0(\overline{\lambda}) \chi \partial_x^\alpha \|_{L^2 \to L^2}, \qquad |\alpha| = 1. 
\end{aligned}
\end{equation}

Noting that we can make the same definition \eqref{defn of matrix op} for $\imag \lambda < 0$, and then using \eqref{cutoff resolv final est}, \eqref{adjoint formulas}, and the proof strategy of Proposition \ref{continuation by comp}, we get
\begin{equation} \label{same est below axis}
S_{R_2}R_B(\lambda)(u_0,u_1) \lesssim \begin{cases} e^{C_1 |\real \lambda |} & |\real \lambda | > \varepsilon_0, \text{ } \imag \lambda < 0, \\
1 + |\lambda|^{n-2} |\log \lambda | & 0< |\real \lambda | \le \varepsilon_0, \text{ } \imag \lambda < 0.
\end{cases}
\end{equation}

\section{Proof of local energy decay} \label{proof of local energy decay}
We now give the proof of Theorem \ref{decay}, our local energy decay.  The proof proceeds in the spirit of  \cite[Proposition 1.4]{povo99}. The idea is to rewrite the wave propagator using the spectral theorem and Stone's formula, and then make an appropriate contour deformation which is made possible by Proposition \ref{continuation by comp}. 
\begin{proof}[Proof of Theorem \ref{decay}]
Throughout the proof, we use $a \lesssim b$ to denote $a  \le Cb$, where $C> 0$ is a constant that does not depend on $t$ or the initial data $(u_0,u_1)$. If a norm appears without a subscript, it denotes the norm on $(L_R^2)^{n+1}$. 

It it enough to show that  
\begin{equation*}
\| S_{R_2} e^{itB} (u_0,u_1) \| \lesssim \frac{1}{(\log(2 + t))^k} \| (u_0,u_1) \|_{D(B^k)}, \qquad t \ge 0.
\end{equation*}
Moreover, we can replace $\|(u_0,u_1) \|_{D(B^k)}$ by $\| (B - i)^{k}(u_0,u_1) \|_H$ on the right side because the spectral theorem shows that the operators $B^k$ and $(B- i)^{k}$ have the same domain and that the norms $\| (u_0,u_1) \|_{D(B^k)}$ and  $\| (B- i)^{k}(u_0,u_1) \|_H$ are equivalent.

Let $E$ denote the spectral measure associated to $B$, and let $ X = X(t)$ be a parameter which depends on $t$. In the last step of the proof we give the explicit dependence of $X$ on $t$. 

To keep our notation concise, set $F(\lambda) = e^{-it\lambda}(\lambda - i)^{-k}$. The wave propagator may be rewritten as
\[ \begin{split}
e^{-itB} (u_0, u_1) &=  e^{-itB} (B- i)^{-k}(B- i)^{k} (u_0, u_1) \\ 
& = \int_{-\infty}^\infty F(\lambda) dE(\lambda) (B- i)^{-k} (u_0, u_1) \\
& = \left( \int_{-X}^X F(\lambda) dE(\lambda)  +  \int_{|\lambda| \ge X} F(\lambda) dE(\lambda) \right) (B- i)^{k}(u_0,u_1) \\
& = (I_{|\lambda| < X} + I_{|\lambda| \ge X})  (B- i)^{k} (u_0, u_1)
\end{split} \]
We apply $S_{R_2}$ to each of the two integrals and estimate them by separate methods.

To handle $S_{R_2} I_{|\lambda| \ge X} (B- i)^{k}(u_0,u_1)$, let $\mathbf{1}_{\mathbb{R}\setminus [-X,X]}$ denote the indicator function of the set $\mathbb{R}\setminus [-X,X]$. Then, using $\| S_R \|_{H \to (L_R^2)^{n+1} } =1$ and properties of the spectral measure,
\begin{equation} \label{bigger than X}
\begin{aligned}
\|S_{R_2} I_{|\lambda| \ge X}  \|_{H \to (L_R^2)^{n+1}} &\le \| I_{|\lambda| \ge X} \|_{H \to H}  \\
&\le \sup_{|\lambda | \ge X} \left|F(\lambda) \mathbf{1}_{\mathbb{R}\setminus [-X,X]}(\lambda) \right| \\
&\lesssim  X^{-k}.
\end{aligned} 
\end{equation}

To estimate $S_{R_2} I_{|\lambda| < X} (B -i)^k (u_0, u_1)$, we use Stone's formula, which says that, with respect to strong convergence, the spectral measure may be expressed as
\begin{equation*}
dE(\lambda) = \lim_{\varepsilon \to 0^+} (2\pi i )^{-1}(R_B(\lambda + i \varepsilon) - R_B(\lambda - i \varepsilon))d \lambda.
\end{equation*}
For each $\varepsilon > 0$, we can move $(B- i )^k(u_0,u_1)$ inside the integral. In addition, the boundedness of $S_{R_2}$ allows us to commute it through this strong limit. We get
\[ \begin{split}
2 \pi i S_{R_2} I_{|\lambda| \ge X}  \langle B \rangle^{k} (u_0, u_1) & = \lim_{\varepsilon \to 0^+} \int_{-X}^{X} F(\lambda)S_{R_2}(R_B(\lambda + i \varepsilon) - R_B(\lambda - i \varepsilon)) (B-i)^{k} (u_0, u_1) d\lambda  \\
&= \lim_{\varepsilon \to 0^+} \biggl( \int_{-X}^X F(\lambda) S_{R_2} R_B(\lambda + i \varepsilon) (B-i)^{k} (u_0, u_1) d\lambda \\
&+  \int_{-X}^X F(\lambda) S_{R_2} R_B(\lambda - i \varepsilon) (B-i)^{k} (u_0, u_1) d\lambda \biggr) \\
& = \lim_{\varepsilon \to 0^+} \biggl( \int_{-X+ i \varepsilon}^{X+ i \varepsilon} F(\lambda - i \varepsilon) S_{R_2} R_B(\lambda) (B-i)^{k} (u_0, u_1) d\lambda \\
&+  \int_{-X - i \varepsilon}^{X - i \varepsilon} F(\lambda + i \varepsilon) S_{R_2} R_B(\lambda) (B-i)^{k} (u_0, u_1) d\lambda \biggr) \\
& = \lim_{\varepsilon \to 0^+} \left( I^+(\varepsilon) + I^-(\varepsilon) \right). 
\end{split} \]
The endpoints for the final two integrals indicate that we integrate over the line segments $\{\lambda \pm i \varepsilon : \lambda \in [-X,X]\}$.

As discussed in section \ref{define resolv on homg}, the operator $B$ sends $H_{R_1}$ into $H_{R'}$ for any $R' > R_1$, hence $(B - i)^{k}(u_0,u_1) \in H_{R'}$. Therefore, Proposition \ref{continuation by comp} applies to $S_{R_2} R_B(\lambda)(B-i)^k(u_0, u_1)$. Setting $C_3 \defeq \max \{2C_1, C_2\}$, we perform a contour deformation for $I^+(\varepsilon)$ which has seven segments, $I^+_k = I^+_k(\varepsilon)$, $1 \le k \le 7$, see Figure \ref{plus}. 

\begin{figure}
\labellist
\small 
\pinlabel $M$  at 105 17
\pinlabel $-M$  at -1 17
\pinlabel $-\varepsilon_0$ at 65 18
\pinlabel $-\varepsilon_0$ at 37 18
\pinlabel $-\varepsilon$ at 46 17
\pinlabel $\varepsilon$ at 50 25
\pinlabel $\varepsilon$ at 59 17
\pinlabel $-e^{-C_3X}$ at 46 1
\pinlabel $I^+_1$ at 24 7
\pinlabel $I^+_2$ at 80 7
\pinlabel $I^+_3$ at 0 10
\pinlabel $I^+_4$ at 105 10
\pinlabel $I^+_5$ at 47 10
\pinlabel $I^+_6$ at 58 10
\pinlabel $I^+_7$ at 58 26
\endlabellist
\includegraphics[width=12cm]{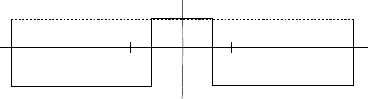}
 \caption{The contour deformation for $I^+(\varepsilon)$.}
 \label{plus}
\end{figure}

\begin{figure}[h]
\labellist
\small 
\pinlabel $M$  at 101 25
\pinlabel $-M$  at 2 25
\pinlabel $-\varepsilon$ at 50 17
\pinlabel $-e^{-C_3X}$ at 46 1
\pinlabel $I^-_1$ at 15 7
\pinlabel $I^-_2$ at 0 10
\pinlabel $I^-_3$ at 104 10
\endlabellist
\includegraphics[width=12cm]{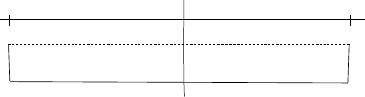}
 \caption{The contour deformation for $I^-(\varepsilon)$.}
 \label{minus}
\end{figure}

We use \eqref{RB est} to estimate the integral over each segment, and omit the factor $\|(B-i)^k(u_0,u_1)\|_H$ that should appear on the right side of each inequality:
\begin{equation} \label{I plus}
\begin{aligned}
\| I_1^+(\varepsilon) \|,\|  I_2^+(\varepsilon) \| &\lesssim Xe^{-te^{-C_3X} + C_1 X},  \\ 
\| I^+_3(\varepsilon) \|, \|I^+_4(\varepsilon) \| &\lesssim (\varepsilon + e^{-C_3 X}) X^{-k} e^{C_1 X}\\
\| I^+_5(\varepsilon)\|,  \|I^+_6(\varepsilon)\|&\lesssim \int_{-e^{-C_3 X}}^{\varepsilon} e^{\varepsilon t} | \log |r||dr ,\\
\| I^+_7(\varepsilon) \| &\lesssim \int_{-\varepsilon}^\varepsilon e^{\varepsilon t} |\log |r| | dr , \\
\end{aligned}
\end{equation}

To handle $I^{-}(\varepsilon)$, we deform it into three segments, $I_k^- = I_k^-(\varepsilon)$, $1 \le k \le 3$. Using \eqref{same est below axis}, and again omitting the factor $\| (B-i)^k (u_0,u) \|_H$, we have
\begin{equation} \label{I minus}
\begin{aligned}
\| I^-_1(\varepsilon) \| &\lesssim X e^{-te^{C_3X} + C_1X}, \\ 
\| I^-_2(\varepsilon) \|,\| I^-_3(\varepsilon) \|  &\lesssim X^{-k} e^{(C_1-C_3) X}, \\
\end{aligned}
\end{equation}
Taking $\varepsilon \to 0^+$ and  using the bounds from \eqref{bigger than X}, \eqref{I plus}, and \eqref{I minus}, we get
\begin{equation} \label{penultimate bound}
\|S_{R_2}e^{itB} (u_0,u_1)\| \lesssim \left( Xe^{-te^{-C_3X} + C_1 X}+ X^{-k} \right)\| (B- i)^k(u_0,u_1) \|_{D(B^k)}.
\end{equation}

To finish the proof, set $X(t) = (2C_3)^{-1}\log (2 + t)$. We have,
\begin{equation} \label{log bound above}
\begin{aligned}
\log(2 + t) &\lesssim t(t + 2)^{-1/2} - C_1(2C_3)^{-1}\log(2+t)  \\
& =  te^{-C_3X} - C_1 X, \qquad t \to \infty.
\end{aligned}
\end{equation}
Furthermore, for any $C > 0$,
\begin{equation} \label{lhospital}
xe^{-Cx} \lesssim x^{-k}, \qquad x >0. 
\end{equation}
Plugging the expression for $X(t)$ into \eqref{penultimate bound} and estimating using \eqref{log bound above} and \eqref{lhospital} completes the proof. \\
\end{proof}
\appendix
\section{Proofs of self-adjointness}
\label{proofs of self-adjointness}
\begin{proposition} \label{L self adj}
The operator $L= -c^2(x) \Delta : L^2_c(\mathbb{R}^n) \to L_c^2(\mathbb{R}^n)$ with domain $D(L) = H^2(\mathbb{R}^n)$  is self-adjoint. 
\end{proposition}
\begin{proof}
We need to show that $D(L^*) = H^2(\mathbb{R}^n)$, and that 
\begin{equation} \label{L across prod}
\langle Lu,v \rangle_{L^2_c} =\langle u,Lv \rangle_{L^2_c}, \qquad u,v \in H^2(\mathbb{R}^n).
\end{equation}
First, we show that $H^2(\mathbb{R}^n) \subseteq D(L^*)$, and that \eqref{L across prod} holds. Let $u,v \in H^2(\mathbb{R}^n)$. We use the fact that integration by parts holds for functions $u, v \in H^2(\mathbb{R}^n)$,
\[ \begin{split}
\langle u,Lv \rangle_{L^2_c} &\defeq \int u (\overline{-c^2 \Delta v}) c^{-2} \\ 
 &= -\int u \Delta \overline{v} \\
 &= -\int \Delta u \overline{v} \\
 &= \int -c^2\Delta u \overline{v}c^{-2} \\
 &= \langle Lu,v \rangle_{L^2_c}. \\
\end{split} \]

To see that $D(L^*) \subseteq H^2(\mathbb{R}^n)$, suppose $u \in D(L^*)$. By definition, there exists a unique $\tilde{u} \in L^2_c(\mathbb{R}^n)$ so that for all $v \in H^2(\mathbb{R}^n)$, 
\begin{equation} \label{adj domain prop}
\langle u,Lv \rangle_{L^2_c} = \langle \tilde{u},v \rangle_{L^2_c}.
\end{equation}

Let $\mathcal{F}$ denote the Fourier transform. Using the Fourier transform characterization of $u \in H^2(\mathbb{R}^n)$, it suffices to show there exists $C>0$ so that for all $\varphi \in C^\infty_0(\mathbb{R}^n)$
\begin{equation} \label{suff for H two}
\left| \langle (1 + |\cdot|^2) \mathcal{F}u, \varphi \rangle_{L^2} \right| \le C \| \varphi \|_{L^2},  
\end{equation}
By properties of $\mathcal{F}$,
\[ \begin{split}
\left| \langle (1 + |\cdot|^2) \mathcal{F}u, \varphi \rangle_{L^2} \right| &= \left|\langle u,\mathcal{F}^{-1}(1 + | \cdot |^2)  \varphi \rangle_{L^2} \right| \\ 
&= \left|\langle u,\mathcal{F}^{-1}\varphi \rangle_{L^2} + \langle u, L \mathcal{F}^{-1}\varphi \rangle_{L^2_c} \right|\\
&=  \left|\langle u,\mathcal{F}^{-1}\varphi \rangle_{L^2} + \langle \tilde{u}, \mathcal{F}^{-1}\varphi \rangle_{L^2_c} \right|\\
 &\le \| u \|_{L^2} \|\mathcal{F}^{-1}\varphi \|_{L^2}  +\|c^2\|_{L^\infty} \| \tilde{u} \|_{L^2} \|\mathcal{F}^{-1}\varphi \|_{L^2}. \\
& \le C\|\varphi\|_{L^2},
\end{split} \]
where $C >0$ depends on $u$, $\tilde{u}$, $c$ and $\|\mathcal{F}^{-1}\|_{L^2 \to L^2}$. This establishes \eqref{suff for H two} and completes the proof.
\\ 
\end{proof}
\begin{proposition}
The operator
\begin{equation*}
B \defeq \begin{bmatrix} 0 & iI \\ -iL  & 0 \end{bmatrix}: H \to H,
\end{equation*}
 with respect domain 
\begin{equation*}
D(B) \defeq \{(u_0,u_1) \in H : \Delta u_0 \in L^2(\mathbb{R}^n),  u_1 \in H^1(\mathbb{R}^n)\}.
\end{equation*} 
is self-adjoint.
\end{proposition}
\begin{proof}
Suppose $(u_1, v_1),(u_2,v_2) \in D(B)$. We compute
\[\begin{split}
\langle (u_1, v_1), B(u_2,v_2) \rangle_H &= \langle \nabla u_1, i\nabla v_2 \rangle_{L^2} + \langle v_1, -i L u_2 \rangle_{L_c^2} \\ 
&= - \langle \Delta u_1, iv_2  \rangle_{L^2} + \langle \nabla v_1, i \nabla u_2  \rangle_{L^2} \\
&= \langle -iL u_1, v_2  \rangle_{L_c^2} + \langle i \nabla v_1,  \nabla u_2  \rangle_{L^2} \\
&= \langle (i v_1, -iLu_1) , (u_2, v_2) \rangle_H \\
& = \langle B(u_1, v_1) , (u_2, v_2) \rangle_H. \\
\end{split}\]
It remains to show that $D(B^*) \subseteq D(B)$. To this end, suppose $(u,v) \in D(B^*)$. Then there exists unique $(\tilde{u}, \tilde{v}) \in H$ such that for all $(u_1, v_1) \in D(B)$,
\begin{equation} \label{definition domain of B}
\langle (u,v) , B(u_1, v_1) \rangle_H = \langle (\tilde{u}, \tilde{v}), (u_1, v_1) \rangle_H.
\end{equation}
This implies 
\begin{align}
\langle v, -iLu_1 \rangle_{L^2_c} &= \langle \tilde{u}, u_1 \rangle_{\dot{H}^1}, \qquad u_1 \in \dot{H^1}, \text{ } \Delta u_1 \in L^2, \label{v property} \\
\langle u, iv_1 \rangle_{\dot{H}^1} &= \langle \tilde{v}, v_1 \rangle_{L^2_c}, \qquad v_1 \in H^1. \label{u property}
\end{align}
To show $(u,v) \in D(B)$, it suffices to show 
\begin{align}
(1 + |\xi|^2)^{\frac{1}{2}} \mathcal{F}v &\in L^2(\mathbb{R}^n), \label{belongs to H1} \\
 \qquad \sum_{j = 1}^n \xi_j \mathcal{F}(\partial_{x_j} u) &\in L^2(\mathbb{R}^n) \label{Laplace in L2}. 
\end{align}
Observe that \eqref{Laplace in L2} ensures that the distributional Laplacian of $u$ belongs to $L^2(\mathbb{R}^n)$, according to the calculation
\[ \begin{split}
\int u(x) \Delta \overline{\varphi(x)}dx &= - \int \sum_{j=1}^n \partial_{x_j} u(x) \partial_{x_j} \overline{ \varphi(x)}dx \\ 
&= \int \sum_{j = 1}^n i \xi  \mathcal{F}(\partial_{x_j}u)(\xi)\overline{\mathcal{F}\varphi(\xi)} d\xi , \qquad  \varphi \in C_0^\infty(\mathbb{R}^n).    
\end{split} \]

To show \eqref{belongs to H1}, we first demonstrate that the subspace $\{|\xi| \mathcal{F}\varphi : \varphi \in \mathcal{S}(\mathbb{R}^n)\}$ is dense in $L^2(\mathbb{R}^n)$. Suppose that $u \in L^2(\mathbb{R}^n)$ has the property 
\begin{equation*}
\int |\xi | u(\xi) \overline{\mathcal{F}\varphi(\xi)} d\xi = 0, \qquad \varphi \in \mathcal{S}(\mathbb{R}^n). 
\end{equation*}
This implies $|\cdot| u \in L^1_{loc}(\mathbb{R}^n)$ has the property
\begin{equation*}
\int |\xi| u(\xi) \eta(\xi) d\xi= 0, \qquad \eta \in C^\infty_0(\mathbb{R}^n). 
\end{equation*}
So, almost everywhere, we must have $|\xi |u(\xi) =0$, which in turn requires $u = 0$ in $L^2(\mathbb{R}^n)$. This confirms that $\{|\xi| \mathcal{F}\varphi : \varphi \in \mathcal{S}(\mathbb{R}^n)\}$ is dense in $L^2(\mathbb{R}^n)$. 

Now, for all $\varphi \in \mathcal{S}(\mathbb{R}^n)$, \eqref{v property} says
\[ \begin{split}
\left|\langle |\cdot| \mathcal{F}v, |\cdot| \mathcal{F}\varphi \rangle_{L^2}\right| &= \left | \int v(x)  \overline{-\Delta \varphi(x)} dx \right|  \\
&= \left| \langle v, -iL (i\varphi)  \rangle_{L^2_c}  \right| \\
&= \left| \langle \tilde{u}, i\varphi  \rangle_{\dot{H}^1} \right| \\ 
&\le \| \tilde{u} \|_{\dot{H}^1} \| \varphi \|_{\dot{H}^1} \\
&= \| \tilde{u} \|_{\dot{H}^1} \| |\cdot| \mathcal{F}\varphi \|_{L^2}.
\end{split} \]
This shows $| \cdot | \mathcal{F}v \in L^2(\mathbb{R}^n)$ and so we have proved \eqref{belongs to H1}

To finish, we show \eqref{Laplace in L2}. For $\varphi \in \mathcal{S}(\mathbb{R}^n)$, we use \eqref{u property} to calculate
\[ \begin{split}
\left| \left \langle \sum^n_{j =1} \xi_j \mathcal{F}(\partial_{x_j} u), i\mathcal{F}\varphi \right\rangle_{L^2} \right| &=\left| \langle \nabla u, i\nabla\varphi \rangle_{L^2} \right| \\
&= \left| \langle u, i\varphi \rangle_{\dot{H}^1} \right| \\
&= \left| \langle \tilde{v}, \varphi \rangle_{L^2_c} \right| \\
&\le \|c^{-2} \|_{L^\infty} \|\tilde{v} \|_{L^2} \| \varphi \|_{L^2}.
\end{split} \]
This establishes \eqref{Laplace in L2} and completes the proof. \\
\end{proof}

\end{document}